\documentclass[10pt,a4paper]{amsart}%

\usepackage{amsfonts,amsmath,amssymb}
\usepackage{amssymb,amsthm,amsxtra}
\usepackage[usenames]{color}
\usepackage{amscd}
\usepackage{amsthm}
\usepackage{amsfonts}
\usepackage{amssymb}
\usepackage{amsmath}
\usepackage{graphicx}
\usepackage{hyperref}
\usepackage{enumerate}%
\usepackage[all]{xy}
\usepackage[usenames,dvipsnames]{xcolor}
\usepackage{mathrsfs}
\usepackage[margin=1.45in]{geometry}

 \newtheorem*{corollary*}{Corollary}
 \newtheorem*{construction*}{Construction}
 \newtheorem*{definition*}{Definition}
 \newtheorem*{notation*}{Notation}
 \newtheorem*{lemma*}{Lemma}
 \newtheorem*{theorem*}{Theorem}
 \newtheorem*{remark*}{Remark}
 \newtheorem*{example*}{Example}
 \newtheorem*{conjecture*}{Conjecture}
 \newtheorem*{condition*}{Condition}
 \newtheorem*{result*}{Result}
 \newtheorem*{property*}{Property}

 \newtheorem*{cor*}{Corollary}
 \newtheorem*{const*}{Construction}
 \newtheorem*{defn*}{Definition}
 \newtheorem*{notn*}{Notation}
 \newtheorem*{lem*}{Lemma}
 \newtheorem*{thm*}{Theorem}
 \newtheorem*{rem*}{Remark}
 \newtheorem*{exm*}{Example}
 \newtheorem*{conj*}{Conjecture}

 \newtheorem{lemma}{Lemma}[subsection]

 \newtheorem{notation}[lemma]{Notation}

 \newtheorem{thm}[lemma]{Theorem}
 
 \newtheorem{lem}[lemma]{Lemma}
 \newtheorem{defn}[lemma]{Definition}
 \newtheorem{notn}[lemma]{Notation}
 \newtheorem{cor}[lemma]{Corollary}

 \newtheorem{introtheorem}{Theorem}

 \newtheorem{introcor}[introtheorem]{Corollary}
 
 \newtheorem{introconj}[introtheorem]{Conjecture}

\newcommand{\al}{\alpha}
\newcommand{\alp}{\alpha}

\newcommand{\eps}{\varepsilon}

\newcommand{\End}{\operatorname{End}}

\newcommand{\bR}{\mathbb{R}}

\newcommand{\Sc}{\cS}
\newcommand{\cS}{\mathcal{S}}

\providecommand{\fb}{\mathfrak{b}}
\providecommand{\fh}{\mathfrak{h}}
\providecommand{\fg}{\mathfrak{g}}
\providecommand{\ft}{\mathfrak{t}}
\providecommand{\fu}{\mathfrak{u}}
\providecommand{\fz}{\mathfrak{z}}

\providecommand{\cO}{\mathcal{O}}
\providecommand{\cN}{\mathcal{N}}
\providecommand{\cU}{\mathcal{U}}

\providecommand{\g}{\mathfrak{g}}

\newcommand{\simpAr}[2][r]{%
\ar@{}[#1]|-*[@]_{#2}%
}

\newcommand{\RamiA}[1]{{{#1}}}
\newcommand{\DimaA}[1]{{{#1}}}

\newcommand{\proofend}{\hfill$\Box$\smallskip}

\begin{document}

\author{Avraham Aizenbud}
\address{Avraham Aizenbud,
Faculty of Mathematics and Computer Science, Weizmann
Institute of Science, POB 26, Rehovot 76100, Israel }
\email{aizenr@gmail.com}
\urladdr{http://www.wisdom.weizmann.ac.il/~aizenr/}
\author{Dmitry Gourevitch}
\address{Dmitry Gourevitch, Faculty of Mathematics and Computer Science, Weizmann
Institute of Science, POB 26, Rehovot 76100, Israel }
\email{dimagur@weizmann.ac.il}
\urladdr{http://www.wisdom.weizmann.ac.il/~dimagur}

%
%
%
%
%
%
%
%

\date{\today}
\title{Vanishing of certain equivariant distributions on spherical spaces}
\maketitle
\begin{abstract}
\DimaA{
We prove vanishing of $\fz$-eigen distributions on a split real reductive group which change according to  a non-degenerate character under the left action of the unipotent radical of the Borel subgroup, and are equivariant under the right action of a spherical subgroup.}

This is a generalization of a result by Shalika, that concerned the group case. Shalika's result was crucial in the proof of his multiplicity one theorem. We view our result as a step in the study of multiplicities of quasi-regular representations on spherical varieties.

As an application we prove non-vanishing of \DimaA{spherical} Bessel functions.
\end{abstract}
\section{Introduction}\label{sec:intro}

\subsection{Main results}

In this paper we prove the following generalization of Shalika's result \cite[\S 2]{Shal}.
\begin{introtheorem}\label{thm:main}
\DimaA{Let $G$ be a split real reductive group and $H$ be its spherical subgroup. Let $U$ be the unipotent radical of a Borel subgroup $B$ of $G$. Let $\psi$ be a non-degenerate character of $U$ and $\chi$ be a character of $H$. Let $Z$ be the complement to the union of open $B\times H$-double cosets in $G$. Let $\fz$ be the center of the universal enveloping algebra of the Lie algebra $\g$ of $G$.

Then there are no non-zero $\fz$-eigen $(U\times H,\psi\times \chi)$-equivariant distributions supported on $Z.$}
\end{introtheorem}

This result in the group case (\cite[\S 2]{Shal}) was crucial in the proof of Shalika's multiplicity one theorem.

Our proof begins by applying the technique used by Shalika. However, this technique was not enough for this generality and we had to complement it by using integrability of the singular support, as in \cite{AGAMOT}.

Theorem \ref{thm:main} provides a new tool for the study of the multiplicities of the irreducible quotients of the quasi-regular representation of $G$ on Schwartz functions on $G/H$, see \S\S \ref{subsec:qr} below for more details.

\subsection{Non-vanishing of spherical Bessel functions}
Another application of Theorem \ref{thm:main} is to the study of spherical Bessel distributions and functions.
\begin{defn}
Let $G$ be a split real reductive group, and $H \subset G$ be a spherical subgroup.
Let $(\pi,V)$ be a (smooth) irreducible admissible representation of $G$. Let $\phi$ be an \DimaA{$(H,\chi)$-equivariant} continuous functional on $V$ and $v$ be  a $(U,\psi)$-equivariant continuous functional on the contragredient representation $\tilde V$. Define the \emph{spherical Bessel distribution} by
$$\xi_{v,\phi}(f):=\langle  v, \pi^*(f) \phi \rangle. $$

Define the \emph{spherical Bessel function} to be the restriction $j_{v,\phi}:=\xi_{v,\phi}|_{\DimaA{G} - Z}$.
\end{defn}
It is well-known that $j_{v,\phi}$ is a smooth function.

Theorem \ref{thm:main} easily implies the following corollary.
\begin{introcor}\label{cor:SpherChar}
Suppose that \DimaA{$v$ and $\phi$ are non-zero.} Then $j_{v,\phi} \neq 0$.
\end{introcor}

\DimaA{
\subsection{Non-archimedean analogs}$\,$\\
Over non-archimedean fields, the universal enveloping algebra does not act on the distributions. However, the Bernstein's center $\End_{G \times G}(\Sc(G))$
does act.
In \cite{AGS_Z} we study this action in details. In \cite{AGK} we prove, using \cite[Theorem A]{AGS_Z},  analogs of Theorem \ref{thm:main}
and Corollary \ref{cor:SpherChar} for non-archimedean fields of characteristic zero. These analogs are somewhat weaker for general spherical pairs, but are of the same strength for the group case and for Galois symmetric pairs.
The group case of the non-archimedean counterpart of Corollary \ref{cor:SpherChar} was proven before in \cite[Appendix B]{LM}.

}
\subsection{Relation with multiplicities in regular representations of symmetric spaces}\label{subsec:qr}

Let $(G,H)$ be a symmetric pair of real reductive groups. Suppose that $G$ is quasi-split and let $B\subset G$ be a Borel subgroup. Let $k$ be the number of open $B$-orbits on $G/H$.

Theorem \ref{thm:main} can be used in order to study the following conjecture.

\begin{introconj}
Let $(\pi,V)$ be a (smooth) irreducible admissible representation of $G$. Then the dimension of the space \DimaA{$(V^*)^{H}$ of $H$-invariant} continuous functionals on $V$ is at most $k$. In particular, any complex reductive symmetric pair is a Gelfand pair.
\end{introconj}

We suggest to divide this conjecture into two cases
\begin{itemize}
\item $\pi$ is non-degenerate, i.e. $\pi$ has a non-zero continuous $(U,\psi)$-equivariant functional for some non-degenerate character $\pi$ of $U$
\item $\pi$ is degenerate.
\end{itemize}

In the first case, the last conjecture follows from the following one

\begin{introconj}
 Let $U$ be the unipotent radical of $B$ and let $\psi$ be its non-degenerate character. Let $\fz$ be the center of the universal enveloping algebra of the Lie algebra $\g$ of $G$. Let $\lambda$ be a character of $\fz$.

Then the dimension of the space of $(\fz,\lambda)$-eigen $(U,\psi)$-equivariant distributions on $G/H$ does not exceed $k$.
\end{introconj}

We believe that  Theorem \ref{thm:main} can be useful in approaching this conjecture, since it allows to reduce the study of distributions to the union of open $B$-orbits.


\subsection{Acknowledgements}
We thank Erez Lapid for pointing out to us the application to non-vanishing of spherical Bessel functions. \RamiA{We also thank Vladimir Hinich, Omer Offen and Eitan Sayag for useful discussions.}
D.G. was partially supported by ISF grant 756/12 and a Minerva foundation grant.
A.A. was partially supported by NSF grant DMS-1100943 and ISF grant 687/13;
\section{Preliminaries}

\subsection{Conventions}

\begin{itemize}
\item By an algebraic manifold we mean a smooth real algebraic variety.
\item We will use capital Latin letters to denote Lie groups and the corresponding Gothic letters to denote their Lie algebras.
\item Let a Lie group $G$ act on a smooth manifold $M$. For a vector $v\in \fg$ and a point $x \in M$ we will denote by $vx\in T_xM$ the image of $v$ under the differential of the action map $g \mapsto gx$. Similarly, we will use the notation $\fh x$, for any subspace $\fh \subset \fg$.
\item We denote by $G_x$ the stabilizer of $x$ in $G$ and by $\fg_x$ its Lie algebra.
\end{itemize}

\subsection{Tangential and transversal differential operators}\label{subsec:trans}

In this subsection we shortly review the method of \cite[\S 2]{Shal}. For a more detailed description see \cite[\S\S 2.1]{JSZ}.

\begin{defn}
Let $M$ be a smooth manifold and $N$ be a smooth submanifold.
\begin{itemize}
\item A vector field $v$ on $M$ is called \emph{tangential} to $N$ if for every point $p\in N, \, v_p\in T_pN$ and \emph{transversal} to $N$ if for every point $p\in N, \, v_p\notin T_pN$.

\item A differential operator $D$ is called \emph{tangential} to $N$ if every point $p\in N$ has an open neighborhood $U_x \subset N$ such that $D|_{U_x}$ is a finite sum of differential operators of the form $\phi V_1\cdot ...\cdot V_r$ where $\phi$ is a smooth function on $U_x, \, r\geq 0$, and $V_i$ are vector fields on $U_x$ tangential to $U_x \cap N$.
\end{itemize}
\end{defn}

\begin{lem}[cf. the proof of {\cite[Proposition 2.10]{Shal}}]\label{lem:TransTan}
Let $M$ be a smooth manifold and $N$ be a smooth submanifold. Let $D$ be a differential operator on $M$ tangential to $N$ and $V$ be a vector field  on $M$ transversal to $N$. Let $\eta$ be a distribution on $M$ supported in $N$ such that $D\eta=V\eta$. Then $\eta=0$.
\end{lem}

\subsection{Singular support}\label{subsec:SS}

Let $M$ be an algebraic manifold and $\eta$ be a distribution on $M$.
The \emph{singular support} of $\eta$ is defined to be the singular support of the D-module generated by $\eta$ and denoted $SS(\eta) \subset T^*M$.

We will shortly review the properties of the singular support that are most important for this paper.
For more detailed overview we refer the reader to \cite[\S\S 2.3 and Appendix B]{AGAMOT}.

\begin{notn} For a point $x\in M$
\begin{itemize}
\item we denote by $SS_x(\eta)$ the fiber of $x$ under the natural projection $SS(\eta)\to M$,
\item for a submanifold $N \subset M$ we denote by $CN_N^M\subset T^*M$ the conormal bundle to $N$ in $M$, and by $CN_{N,x}^M$ the conormal space at $x$ to $N$ in $M$.
\end{itemize}
\end{notn}

\begin{lem}[{See e.g. \cite[Fact 2.3.9 and Appendix B]{AGAMOT} }] \label{Ginv}
Let an algebraic group $G$ act on   $M$. Suppose that $\eta$ is $G$-equivariant. Then
$$SS(\eta) \subset \{(x,\phi) \in T^*M \, | \, \forall \alpha \in
\g, \, \phi(\alpha(x)) =0\}=\bigcup_{x\in M}CN_{Gx}^M.$$
\end{lem}


\begin{lem} \label{lem:nilp}
Let $G$ be a real reductive group, $\cN \subset \fg^*$ be the nilpotent cone and $\fz$ be the center of the universal enveloping algebra $\cU(\g)$.
Let $\xi$ be a $\fz$-eigen distribution on $G$. Identify $T^*G$ with $G\times \fg^*$ using the right action.
Then $SS(\xi)\subset G \times \cN$.
\end{lem}
This lemma is well-known but we will prove it here for the benefit of the reader.
\begin{proof}

Consider the standard filtrations on $\cU(\fg)$ and on the ring of differential operators $D(G)$. Consider $\fg$ as the space of left-invariant vector fields on $G$. Then the natural map $\RamiA{i:}\cU(\fg)\to D(G)$ is a morphism of filtered algebras.
\RamiA{We have a commutative diagram
 $$\xymatrix{\parbox{30pt}{$Gr\cU(\fg)$}\ar@{->}^{Gr(i)}[r]\ar@{<-}_{\pi_U}[d] &
 \parbox{20pt}{$Gr D(G)$}\ar@{<-}_{\pi_D}[d]\\
 \parbox{20pt}{$S(\fg)$ }\ar@{->}^{\bar i}[r] &
 \parbox{40pt}{$\cO(T^*G)$,}}
 $$

Where $\bar i$ and ${\pi_U}$ are the algebra homomorphisms which extend the natural embeddings $\g \to Gr\cU(\fg)$ and $\g \to \cO(T^*G)$, and $\pi_D$ is the algebra homomorphism which extends the natural embedding
of vector fields on $G$ into  $D(G)$. By  the PBW theorem the vertical maps are isomorphisms. This implies that $Gr(i)$ is an embedding, and thus the filtration on $\cU(\g)$ is the one induced from   $D(G)$ using the embedding $i$. Therefore we have the following commutative diagram

$$  \xymatrix{ \cU(\fg)\ar@{->}^{i}[r]\ar@{->}^{\sigma_{U}}[d] &
D(G)\ar@{->}^{\sigma_{D}}[d]\\
\parbox{30pt}{$Gr\cU(\fg)$}\ar@{->}^{Gr(i)}[r]\ar@{<-}_{\pi_U}[d] &
 \parbox{20pt}{$Gr D(G)$}\ar@{<-}_{\pi_D}[d]\\
 S(\fg)\ar@{->}^{\bar i}[r] &
\cO(T^*G),}
$$
where ${\sigma_{U}}$ and $\sigma_D$ are the (nonlinear) symbol maps. Note that the map ${\bar i}$ is a section of the restriction map $r:\cO(T^*G)\to \cO(T_e^*G)\cong\cO(\g^*)\cong S(\g).$
}

%
%
%

In order to prove the lemma it is enough to show that $SS_{e}(\xi)\subset \cN$.
Note that $\cN$ is the zero set the ideal $I\subset S(\fg)=\cO(\fg^*)$ generated by all homogeneous non-constant $\fg$-invariant polynomials. We have to show that for any homogeneous \RamiA{$p \in S(\g)^\g$ of degree $d>0,$}  there exists \RamiA{non-constant} $u\in \fz$ such that \RamiA{$r(\pi_D^{-1}(\sigma_D(i(u))))=p$, or equivalently (in view of the above) that $\sigma_{U}(u)=\pi_U(p)$. Let $s:S(\g)\to\cU(\g)$ be the symmetrization map.   It is easy to see that $\sigma^{d}_{U}(s(p))=\pi_U(p)$, where $\sigma^{d}_{U}$ denotes the $d$'s symbol.  Since $\pi_U$ is an isomorphism this implies that $\sigma^{d}_{U}(s(p)) \neq 0$ and thus $\sigma_{U}(s(p))=\sigma^{d}_{U}(s(p))=\pi_U(p)$. This implies the assertion.

}
\end{proof}

\begin{thm}[Integrability theorem, cf. \cite{Gab,GQS,KKS, Mal}]
 The $SS(\eta)$ is a coisotropic subvariety of  $T^*M$.
\end{thm}

This theorem implies the following corollary (see \cite[\S 3]{Aiz} for more details).

\begin{cor}\label{cor:WCI}
Let $N\subset M$ be a closed algebraic submanifold. Suppose that  $\xi$  is supported in $N$. Suppose further that for any $x \in N$, we have $CN_{N,x}^M \nsubseteq SS_x(\eta)$. Then $\eta=0$.
\end{cor}

\section{Proof of the main result}

\subsection{Sketch of the proof}


\DimaA{We decompose $G$ into $B\times H$-double cosets. Each double coset} $\cO$ we decompose $\cO=\cO_s \cup \cO_c$ in a certain way.
 We prove the required vanishing \DimaA{coset by coset},  using Shalika's method (see \S\S\ref{subsec:trans}) for $\cO_s$ and singular support analyses (see \S\S\ref{subsec:SS}) for $\cO_c$.

\subsection{Notation and lemmas}

\begin{notation}$\,$
\begin{itemize}
\item Fix a torus $T\subset B$ and let $\ft \subset \fb$ denote the corresponding Lie  algebras. Let $\Phi$ denote the root system, $\Phi^+$ denote the set of positive roots and $\Delta\subset \Phi^+$ denote the set of simple roots. For $\alp\in \Phi$ let $\g_{\alpha} \subset \g$ is the root space corresponding  to ${\alpha}$.
 \item Let $C\in \fz$ denote the Casimir element.

\item We choose $E_{\alp}\in \fg_{\alp}$, for any $\alp \in \Phi$ such that $C=\sum_{\alp \in \Phi^+} E_{-\alp}E_{\alp}+D$, where $D$ is in the universal enveloping algebra of the Cartan subalgebra $\ft$.

\item Let $\cO\subset G$ be a $B\times H$-\DimaA{double coset. Consider the left action of $\fu$ on $\cO$ and define} $$\cO_c:=\left \{x \in \cO \, |\, \sum_{\alp \in \Delta} d\psi(E_\alp)E_{-\alp}x \in T_x\cO \right \}$$ and $\cO_s:=\cO\setminus \cO_c$.
\end{itemize}
\end{notation}

We will need the following lemmas, that will be proved in subsequent subsections.
\DimaA{
\begin{lemma}\label{lem:WF}
Let $x\in G$. Let $\xi$ be a $\fz$-eigen $(U\times H,\psi \times \chi)$ equivariant distribution on $G$. Then $SS_x(\xi)\subset CN_{BxH,x}$.
\end{lemma}

\begin{lemma}\label{lem:cOProp}
Let $\cO\subset Z$ be a $B\times H$-double coset. Then $\cO_s \neq \emptyset$.
\end{lemma}
}



\subsection{Proof of Theorem \ref{thm:main}}


Suppose that there exists a non-zero $\fz$-eigen $(U,\psi)$- equivariant distribution $\xi$  supported on $Z$.

For any \DimaA{$B\times H$-double coset $\cO \subset G$}, stratify $\cO_{c}$ to a union of smooth locally closed varieties $\cO_c^i$.

The collection \DimaA{$$\{\cO_c^i \, | \, \cO \text{ is a }B\times H-\text{double coset}\}\cup \{\cO_s \, | \,\cO\text{ is a } B\times H-\text{double coset}\}$$} is a stratification of \DimaA{$G$}. Reorder this collection to a sequence $\{S_i\}_{i=1}^N$ of smooth locally closed subvarieties of \DimaA{$G$} s.t. $U_k:=\bigcup_{i=1}^k S_i$ is open in \DimaA{$G$} for any $1\leq k \leq N$.

Let $k$ be the maximal integer such that $\xi|_{U_{k-1}}=0$. Let $\eta:=\xi|_{U_{k}}$.
We will now show that $\eta=0$, which leads to a contradiction.
\begin{enumerate}[{Case}  1.]
\item $S_{k}=\cO_s$ for some \DimaA{double coset} $\cO$.

Recall that we have the following decomposition of the Casimir element $$C=\sum_{\alp \in \Phi^+} E_{-\alp}E_{\alp}+D$$ Since $\eta$ is  $\fz$-eigen and $(U,\psi)$-equivariant, we have, for some scalar $\lambda$,
$$\lambda \eta=C\eta=\sum_{\alp \in\Phi^+} E_{-\alp}E_{\alp}\eta+D\eta=\sum_{\alp \in \Phi^+} E_{-\alp}d\psi(E_\alp)\eta +D\eta=\sum_{\alp \in \Delta} E_{-\alp}d\psi(E_\alp)\eta +D\eta$$

%


Let $V:= \sum_{\alp \in \Delta} d\psi(E_\alp)E_{-\alp}$ and
$D':=\lambda Id-D$.
We have $V\eta=D'\eta$, and it is easy to see that $D'$ is tangential to $\cO_s$,
and $V$ is transversal to $\cO_s$. Now, Lemma \ref{lem:TransTan} implies $\eta=0$ which is a contradiction.

\item $S_{k}\subset \cO_c$ for some orbit $\cO$.


%
%

By Corollary \ref{cor:WCI} it is enough to show that for any  $x \in S_k$ we have
\begin{equation}\label{eq:NonCoIsot}
CN_{S_k,x}^{\DimaA{G}} \nsubseteq SS_x(\eta).
\end{equation}

By Lemma \ref{lem:WF}, $ SS_x(\eta)\subset CN_{\cO,x}^{\DimaA{G}}$. By Lemma \ref{lem:cOProp}, $S_{k}\subsetneq \cO$, thus $CN_{S_{k},x}^{\DimaA{G}} \supsetneq CN_{\cO,x}^{\DimaA{G}}$ which implies \eqref{eq:NonCoIsot}.



\end{enumerate}
\proofend

\DimaA{
\subsection{Proof of Lemma \ref{lem:WF}}\label{subsec:PfLemWF}

\begin{proof}
Let $\fh$ denote the Lie algebra of $H$  and $ad(x)\fh$ denote its conjugation by $x$. Identify $T_x^*G$ with $\fg^*$ using  multiplication by $x^{-1}$ on the right. Then $$CN_{BxH,x}^G= (\ft+\mathfrak{u}+ad(x)\fh)^\bot.$$
Since $\xi$ is $\mathfrak{u}\times \fh$-equivariant, Lemma \ref{Ginv} implies that  $SS_x(\xi)\subset (\mathfrak{u}+ad(x)\fh)^\bot$.
Since $\xi$ is also $\fz$-eigen, Lemma \ref{lem:nilp} implies that $SS_x(\xi)\subset \cN $, where $\cN \subset \fg^*$ is the nilpotent cone. Now we have
$$SS_x(\xi)\subset(\mathfrak{u}+ad(x)\fh)^\bot \cap \cN =(\ft+\mathfrak{u}+ad(x)\fh)^\bot=CN_{BxH,x}^G.$$
\end{proof}

\subsection{Proof of Lemma \ref{lem:cOProp}}\label{subsec:cOProp}
First we need the following lemmas and notation.

\begin{lemma}\label{lem:dense}
Let $K\subset K_{i} \subset G$ for $i=1,\dots,n$ be algebraic subgroups. Suppose that $K_i$ generate $G$. Let $Y$ be a transitive $G$ space. Let $y \in Y$. Assume that $Ky$  is  Zariski dense in  $K_i y$ for each $i$. Then $Ky$  is Zariski dense in  $Y$.
\end{lemma}
\begin{proof}
By induction we may assume that $n=2$. Let $$O_l:=\underbrace{K_1K_2\cdots K_1K_2 }_{l \text{ times}}y.$$ It is enough to prove that for any $l$ the orbit $Ky$ is dense in $O_l$. Let us prove it by induction on $l$. Suppose that we have already proven that $Ky$ is dense in $O_{l-1}$. Then
$$\overline{Ky}=\overline{K_2y}=\overline{K_2Ky}=\overline{K_2O_{l-1}y}.$$
Thus $Ky$  is dense in $K_2O_{l-1}$. Similarly, $K_2O_{l-1}$  is dense in $K_1K_2O_{l-1}=O_l$.
\end{proof}


\begin{notn}$\,$
\begin{itemize}
\item Let $Y$ denote the symmetric space $G/H$, and $Z'$ denote the image of $Z$ in $Y$.
\item For a simple root $\alp\in \Delta$, denote by $P_{\alp}\subset G$ the parabolic subgroup whose Lie algebra is $\fg_{-\alp}\oplus \fb$.
\end{itemize}
\end{notn}

\begin{lemma}\label{lem:out}
Let $x \in Z'$. Then there exists a simple root $\al \in \Delta$ such that $\g_{-\al}x \nsubseteq    \fb x$.
\end{lemma}

\begin{proof}
Assume the contrary. Then for any $\al \in \Delta, \,T_x P_{\al}x=T_xBx$. Thus $Bx$ is Zariski dense in $P_{\al}x$. By Lemma \ref{lem:dense} this implies that $Bx$ is dense in $Y$, which contradicts the condition $x \in Z'$.
\end{proof}

\begin{proof}[Proof of Lemma \ref{lem:cOProp}]
Note that $\cO_c$ is invariant with respect to the right action of $H$.
Let $\cO'$ denote the image of $\cO$ in $Y$, and let $\cO_c'$ denote  the image of $\cO_c$ in $Y$. Choose $x\in \cO'$ and
let $a:B \to \cO'$ denote the action map. It is enough to show that $a^{-1}(\cO_c')\neq B$.
\begin{multline*}
a^{-1}(\cO_c')=\{b \in B \, | \, \sum \psi(E_{\alp})E_{-\alp}\in \fg_{bx} + \fb\}=
\{b \in B \, | \, \sum \psi(E_{\alp})ad(b^{-1})E_{-\alp} \in \fg_{x} + \fb\}=\\
\{tu \in B \, | \, \sum \psi(E_{\alp})ad(t^{-1})E_{-\alp} \in \fg_{x} + \fb\}=
\{tu \in B \, | \, \sum \psi(E_{\alp})\alp(t)E_{-\alp} \in \fg_{x} + \fb\}
\end{multline*}

By Lemma \ref{lem:out} we can choose $\alp \in \Delta$ such that $\g_{-\al}x \nsubseteq    \fb x$. For any $\eps>0$ there exists $t \in T$ s.t. $\alp(t)=1$ and $\forall \beta \neq \alp \in \Delta$ we have $|\beta(t)|<\eps$. It is easy to see that for $\eps$ small enough, $t \notin  p^{-1}(\cO_c')$.
\end{proof}
}

\end{document}